\documentclass{amsart}
\usepackage{graphicx}

\theoremstyle{plain}
\newtheorem{theorem}{Theorem}[section]
\newtheorem{lemma}[theorem]{Lemma}
\newtheorem{corollary}[theorem]{Corollary}

\theoremstyle{definition}
\newtheorem{definition}[theorem]{Definition}

\theoremstyle{remark}

\newcommand{\del}{\partial}

\newcommand{\R}{\mathbb{R}}

\newcommand{\Z}{\mathbb{Z}}

\newcommand{\RP}{\mathbb{RP}}

\begin{document}

\title{Essential surfaces in Seifert fiber spaces with singular surfaces}

\author[Kalelkar]{Tejas Kalelkar}
\address{Mathematics Department, Indian Institute of Science Education and Research, Pune 411008, India}
\email{tejas@iiserpune.ac.in}

\author[Nair]{Ramya Nair}
\address{Mathematics Department, Indian Institute of Science Education and Research, Pune 411008, India}
\email{nair.ramya@students.iiserpune.ac.in}

\date{\today}


\begin{abstract}
Two-sided incompressible surfaces in Seifert fiber spaces with isolated singular fibers are well-understood. Frohman \cite{Fro} and Rannard \cite{Ran} have shown that one-sided incompressible surfaces in Seifert fiber spaces which have isolated singular fibers are either pseudo-horizontal or psuedo-vertical. We extend their result to characterise essential surfaces in Seifert fiber spaces which may have singular surfaces, i.e., in $S^1$-foliated $3$-manifolds which have fibered model neighbourhoods that are isomorphic to either a fibered solid torus or a fibered solid Klein bottle. 
\end{abstract}

\maketitle

\section{Introduction}
We call a closed $3$-manifold $M$ \emph{irreducible} if every embedded $2$-sphere in $M$ bounds a $3$-ball. The prime decomposition phenomenon of $3$-manifolds allows us to uniquely express every closed $3$-manifold as a connected sum of manifolds which are either irreducible or $S^2 \times S^1$ or $S^2 \tilde{\times} S^1$. The geometrisation of $3$-manifolds further allows us to cut these irreducible summands along a canonical collection of incompressible tori and Klein bottle into pieces which are one of three possible types: either they are Seifert fiber spaces or they are finitely covered by torus bundles or they have interiors which admit a complete hyperbolic metric. Seifert fiber spaces are precisely those compact $3$-manifolds which admit a foliation by circles. They are an important class of $3$-manifolds that are fairly well-understood as they are completely determined by a finite collection of integers. 

Let $D=\{(x, y) \in \R^2 : x^2 + y^2 \leq 1\}$ and let $D^+ = \{(x, y) \in D: x\geq 0\}$. Let $I=[0,1]$, let $\del I = \{0, 1\}$ and let $int(I) = (0,1)$. A \emph{model fibered solid torus} is the monodromy fibering of a solid torus given by $D \times I/\sim_\rho$ or $D^+ \times I/\sim_\rho$. Where $\rho: D \times 1 \to D \times 0$ is a rational rotation map (possibly identity) and $\rho: D^+ \times 1 \to D^+ \times 0$ is the identity map. We call the model fibered solid torus \emph{regular} if $\rho$ is the identity and \emph{non-regular} otherwise. Similarly, a \emph{model fibered Klein bottle} is the monodromy fibering given by $D \times I/\sim_r$ or $D^+ \times I/\sim_r$ where $r: D \times 1 \to D\times 0$ or $r: D^+ \times 1 \to D^+ \times 0$ is the reflection along the $x$-axis. 

Epstein \cite{Eps} has shown that every circle fiber $f$ in a Seifert fiber space has a fibered neighbourhood isomorphic to a model fibered solid torus or a model fibered Klein bottle with $f$ identified with the fiber $0 \times S^1$ in the model. We call $f$ \emph{regular} if it has a fibered neighbourhood that is isomorphic to a regular model fibered solid torus and \emph{singular} otherwise. An isolated singular fiber has a fibered neighbourhood isomorphic to a non-regular model solid torus while the non-isolated singular fibers have fibered neighbourhoods isomorphic to a model solid Klein bottle. The union of the non-isolated singular fibers give a collection of annuli, tori and Klein bottles that we call \emph{singular surfaces}. Let $N$ denote a Mobius strip. The model fibered neighbourhood of a singular surface $C$ is either a solid Klein bottle $N \times I$ (when $C$ is an annulus), $N \times S^1$ (when $C$ is a torus) or $N \tilde{\times} S^1$ (when $C$ is a Klein bottle). A good exposition for Seifert fiber spaces is the survey paper by Scott \cite{Sco} and the preprint of a book by Hatcher \cite{Hat}. A good account of the non-orientable Seifert fibered spaces with singular surfaces is given by Cattabriga et al \cite{CMMN}.
\newline

Let $M$ denote a Seifert fiber space possibly with singular surfaces. Let $S$ denote a properly embedded surface in $M$. We call an embedded disk $D$ in $M$ a \emph{compressing disk of $S$} if $D \cap S = \del D$ which is an essential curve in $S$. We call an embedded disk $D$ in $M$ a \emph{boundary compressing disk of $S$} if $\del D$ is the union of arcs $\alpha$ and $\beta$ where $\alpha = D \cap S$ is an arc in $S$ that is not boundary-parallel and $\beta = D \cap \del M$. For $E \subset \del M$, we say that $S$ is \emph{boundary-compressible with respect to $E$} if $\beta$ lies in $E$. We say that $S$ is \emph{incompressible} or \emph{boundary-incompressible} if it does not have any compressing disks or  boundary-compressing disks respectively. We say that $S$ is \emph{essential} if it is neither $S^2$ nor a boundary-parallel disk and it is both incompressible and boundary-incompressible in $M$. We now define the notions of horizontal, vertical, pseudo-horizontal and pseudo-vertical below.
\begin{definition}
Let $C$ be a fibered annulus, Mobius strip, torus or Klein bottle. We say that a properly embedded arc or curve in $C$ is \emph{horizontal} if it is transverse to the fibration. We say that a curve in $C$ is \emph{vertical} if it is a fiber of the fibration.

We say that a properly embedded surface $S$ in $M$ is \emph{horizontal} if it is transverse to the fibration of $M$ and we call it \emph{vertical} if it is composed of a union of the fibers of $M$.
\end{definition}

\begin{definition}
Let $S$ be a properly embedded surface in a Seifert fiber space $M$. Let $C$ be an isolated singular fiber or a singular surface of $M$ and let $W$ be a model neighbourhood of $C$. We say that the intersection of $S$ with $W$ is \emph{exceptional} if  $S \cap W$ is one of the following:
\begin{enumerate}
\item{A once-punctured non-orientable surface with horizontal boundary in $\del W$ when $C$ is an isolated fiber.}
\item{A pair of pants with vertical boundary in $\del W$ when $C$ is an annulus.}
\item{A once-punctured torus with vertical boundary in $\del W$ when $C$ is a torus.}
\item{A once-punctured Klein bottle with vertical boundary in $\del W$ when $C$ is a Klein bottle.}
\end{enumerate}
\end{definition}

\begin{definition}
Let $S$ be a properly embedded surface in a Seifert fiber space $M$. We say that $S$ is \emph{pseudo-horizontal} if it is a horizontal surface outside model neighbourhoods of the isolated singular fibers and intersects the model neighbourhoods of the isolated singular fibers horizontally or exceptionally. 

We say that $S$ is \emph{pseudo-vertical} if it is a vertical surface outside model neighbourhoods of the singular surfaces and isolated singular fibers and it intersects the model neighbourhoods of the singular surfaces and isolated singular fibers vertically or exceptionally.
\end{definition}

It is well-known that in Seifert fiber spaces without singular surfaces, any two-sided essential surface can be isotoped to become vertical or horizontal (see Hatcher \cite{Hat}). Frohman \cite{Fro} showed that a closed one-sided incompressible surface in an orientable Seifert fiber space with orientable base can be isotoped to become pseudo-horizontal or pseudo-vertical. Rannard \cite{Ran} extended this result to closed incompressible surfaces in non-orientable Seifert fiber spaces without singular surfaces. In this article, we extend Rannard's proof to Seifert fiber spaces which may have boundary and singular surfaces by listing out the incompressible surfaces in a solid Klein bottle. The main theorem we prove in this article is the following:
\begin{theorem}\label{thm: main theorem}
Let $M$ be a Seifert fiber space (possibly with singular surfaces) which has at least one singular fiber and let $S$ be a connected properly embedded essential surface in $M$. Then $S$ can be isotoped to a surface that is pseudo-horizontal or pseudo-vertical. 
\end{theorem}

If $M$ has no singular fibers, i.e., $M$ is an $S^1$-bundle over a surface then Rannard \cite{Ran} has shown that $S$ can be isotoped either to a vertical surface or to a surface that is horizontal outside the model neighbourhood of one regular fiber and intersects the model neighbourhood horizontally or exceptionally. Note that if $M$ has singular surfaces, then it always has a horizontal surface (Theorem 1.2 of \cite{KalNai}).

\section{Properties of Seifert fiber spaces}
In this section, we prove that most Seifert fibered spaces are irreducible and have incompressible boundary and we fix a partition of the manifold into solid tori and solid Klein bottle. 

\begin{figure}
\centering
\def\svgwidth{0.2\columnwidth}
\begingroup%
  \makeatletter%
  \providecommand\color[2][]{%
    \errmessage{(Inkscape) Color is used for the text in Inkscape, but the package 'color.sty' is not loaded}%
    \renewcommand\color[2][]{}%
  }%
  \providecommand\transparent[1]{%
    \errmessage{(Inkscape) Transparency is used (non-zero) for the text in Inkscape, but the package 'transparent.sty' is not loaded}%
    \renewcommand\transparent[1]{}%
  }%
  \providecommand\rotatebox[2]{#2}%
  \newcommand*\fsize{\dimexpr\f@size pt\relax}%
  \newcommand*\lineheight[1]{\fontsize{\fsize}{#1\fsize}\selectfont}%
  \ifx\svgwidth\undefined%
    \setlength{\unitlength}{187.93487767bp}%
    \ifx\svgscale\undefined%
      \relax%
    \else%
      \setlength{\unitlength}{\unitlength * \real{\svgscale}}%
    \fi%
  \else%
    \setlength{\unitlength}{\svgwidth}%
  \fi%
  \global\let\svgwidth\undefined%
  \global\let\svgscale\undefined%
  \makeatother%
  \begin{picture}(1,1.29712626)%
    \lineheight{1}%
    \setlength\tabcolsep{0pt}%
    \put(0,0){\includegraphics[width=\unitlength,page=1]{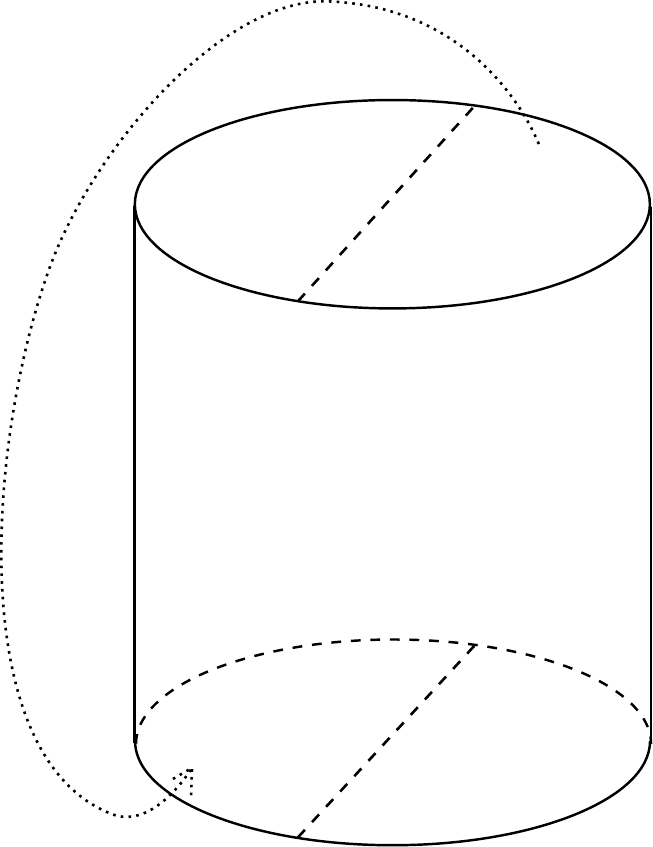}}%
    \put(0.50692463,0.20346287){\makebox(0,0)[lt]{\lineheight{1.25}\smash{\begin{tabular}[t]{l}$m$\end{tabular}}}}%
    \put(0,0){\includegraphics[width=\unitlength,page=2]{CurvesKB.pdf}}%
    \put(0.35881359,0.54157443){\makebox(0,0)[lt]{\lineheight{1.25}\smash{\begin{tabular}[t]{l}$l_1$\end{tabular}}}}%
    \put(0.61380906,0.63321416){\makebox(0,0)[lt]{\lineheight{1.25}\smash{\begin{tabular}[t]{l}$l_2$\end{tabular}}}}%
    \put(0,0){\includegraphics[width=\unitlength,page=3]{CurvesKB.pdf}}%
    \put(0.80747405,0.64537778){\makebox(0,0)[lt]{\lineheight{1.25}\smash{\begin{tabular}[t]{l}$d$\end{tabular}}}}%
    \put(0.24827012,0.65224874){\makebox(0,0)[lt]{\lineheight{1.25}\smash{\begin{tabular}[t]{l}$d$\end{tabular}}}}%
  \end{picture}%
\endgroup%

\caption{The four non-trivial closed curves on a Klein bottle which is obtained from $S^1 \times I$ by identifying $S^1 \times 1$ with $S^1 \times 0$ via a reflection map.}\label{fig: CurvesKB}
\end{figure}

We will reserve the letter $N$ to denote a fibered Mobius strip with monodromy fibering given by $[-1, 1] \times [-1, 1]/ (x, 1) \sim (-x, -1)$. The solid Klein bottle $N \times I$ has an induced fibering which is the same as that of the model fibered solid Klein bottle. The boundary Klein bottle has a fibering given by the cores of the two Mobius strips $N \times 0$ and $N\times 1$ that we denote by $l_1$ and $l_2$ and fibers parallel to  $\del N \times t$ that we denote by $d$. We denote by $m$ the boundary of a meridian disk of the solid Klein bottle (see Figure \ref{fig: CurvesKB}).

Let $M$ be a Seifert fiber space with base space $B$ and projection map $p: M \to B$. Let $M_1$ be the union of disjoint model neighbourhoods of its singular surfaces, i.e., $M_1$ is a (possibly empty) union of $N \times I$, $N \times S^1$ and $N \tilde{\times} S^1$ components. Let $B_1$ denote the projection of $M_1$ on $B$. Let $B_0= \overline{B \setminus B_1}$ denote the base space of the Seifert fiber space $M_0=\overline{M \setminus M_1}$ which has only isolated singular fibers. 

\begin{lemma}\label{lem: boundary-incompressible}
Let $M$ be an irreducible Seifert fiber space with compressible boundary. Then $M$ is a solid torus or a solid Klein bottle.
\end{lemma}
\begin{proof}
Assume that $M$ has a compressible torus or Klein bottle boundary component $T$. Let $D$ be a compressing disk for $T$. Let $N(D) \simeq D \times I$ be a regular neighbourhood of $D$ in $M$. Then $(T\setminus (\del D \times I)) \cup (D \times \del I)$ is an embedded sphere in $M$. As $M$ is given to be irreducible, so this sphere bounds a ball that does not contain $D \times (int (I))$. Hence, $M$ is obtained from two balls by identifying a pair of disks on their boundaries, so $M$ is either a solid torus or a solid Klein bottle.
\end{proof}

The following Theorem \ref{thm: irreducible} is known for Seifert fibered spaces with isolated singular fibers (see Proposition 1.12 of Hatcher \cite{Hat}). We extend this result to Seifert fibered spaces which may have singular surfaces.
\begin{theorem}\label{thm: irreducible}
Let $M$ be a Seifert fibered space (possibly with singular surfaces). Then $M$ is irreducible unless it is $S^2 \times S^1$, $S^2 \tilde{\times} S^1$ or $\RP^3 \# \RP^3$.
\end{theorem}
\begin{proof}
Let $M_0$ and $M_1$ be as described above. If $M_0$ is reducible, then it must have a horizontal reducing sphere. So $M_0$ is either an $S^2$-bundle over $S^1$ or an $S^2$-semibundle over $I$. Hence $M_0$ is closed, and so $M=M_0$ must be either $S^2 \times S^1$, $S^2 \tilde{\times}S^1$ or $\RP^3 \# \RP^3$. See Proposition 1.12 of Hatcher \cite{Hat} for details. We shall henceforth assume that $M_0$ is irreducible. 

Let $P = M_0 \cap M_1$ be a union of tori, Klein bottle and annuli. Let $S$ be a reducing sphere of $M$ that intersects $P$ minimally in its isotopy class. If $S$ does not intersect $P$, then it lies either in $M_0$ or in an $N \times I$, $N\times S^1$ or $N\tilde{\times} S^1$ component of $M_1$ all of which are irreducible. Therefore, $S$ bounds a ball in $M$ contradicting the fact that $S$ is reducing.

Let $D$ be an innermost disk of $S\cap P$ in $S$. If $D$ lies in a solid Klein bottle component $N\times I$ of $M_1$, then it is either a meridian disk or it is parallel to a disk in $\del N \times I$. As $D \subset S$ does not intersect $N \times \del I \subset \del M$, so $D$ can not be a meridian disk. Isotoping $D$ across $\del N \times I$ reduces the number of components of $S \cap P$, which is a contradiction. So, we may assume that $D$ does not lie in an $N\times I$ component of $M_1$.

As $M_0$ has only isolated singular fibers, so it is not a solid Klein bottle. Assume that $M_0$ is not a solid torus, so by Lemma \ref{lem: boundary-incompressible} $M_0$ is an irreducible manifold with incompressible boundary. As fibers are essential in $M_i$, so the annuli components of $P$ which are all vertical are incompressible in $M_i$. Hence, all the components of $P$ are incompressible in $M_i$, and so $\del D$ bounds a disk $D'$ in $P$. And as $M_i$ is irreducible so $D \cup D'$ bounds a ball in $M_i$. There exists an isotopy which sweeps $D$ across this ball and off $D'$ to reduce the number of components of $S\cap P$, which again contradicts the minimality of this intersection. 

Finally, if $M_0$ is a solid torus and $D$ is a compressing disk for $\del M_0$ in $M_0$, then it is a meridian disk of $M_0$ with horizontal boundary. So, $\del M_0$ lies in the interior of $M$ and hence, $M$ is the union of a solid torus $M_0$ and $M_1=N \times S^1$ along their boundary tori. In this case, we claim that $M$ is $S^2 \tilde{\times} S^1$. Let $D_1$ be a horizontal meridian disks of $M_0$. Let $\gamma_1$ be the boundary of $D_1$ in $\del M_1$. Let $\alpha$ be a horizontal straight arc in $N_0 = N \times t_0$ connecting a point of $\gamma_1 \cap \del N_0$ to a point $p$ of $\del N_0$ which is not on $\gamma_1$.  Let $D_2$ be another meridian disk of $M_0$ with a boundary $\gamma_2$ which passes through $p$ and is parallel to $\gamma_1$ on $\del M_1$. Let $A$ be the horizontal annulus in $M_1$ with boundary $\gamma_1 \cup \gamma_2$ obtained by sweeping $\alpha$ across $\gamma_i$. Then $S=D_1 \cup A \cup D_2$ is a horizontal sphere in $M$. The complement of $S$ in $M$ is an $I$-bundle, so $M$ is an $S^2$-bundle over $S^1$ or an $S^2$-semibundle over $I$ (see Corollary 2.6 of \cite{KalNai}). As $M$ is non-orientable, so it must be $S^2 \tilde{\times} S^1$.
\end{proof}

There are no essential surfaces in $S^2 \times S^1$, $S^2 \tilde{\times} S^1$ and $\RP^3 \# \RP^3$, so we shall henceforth assume that $M$ is a Seifert fibered space which is either a solid torus, a solid Klein bottle or an irreducible manifold with incompressible boundary.
\newline

\begin{figure}
\centering
\def\svgwidth{0.6\columnwidth}
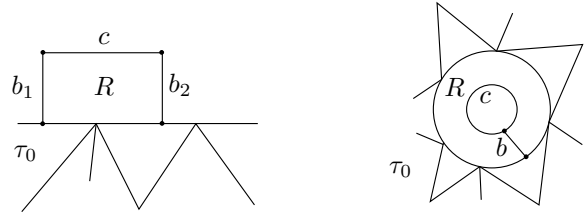
\caption{The CW-complex of the base space $B$ obtained by attaching to $\tau_0$ the rectangles and annuli corresponding to the projection of the model neighbourhoods of the singular surfaces.}\label{fig: CWcomplex}
\end{figure}

\emph{Partition of $M$ into solid tori and solid Klein bottle.}
We first give a CW-complex structure on $B$. Let $\tau_0$ be a simplicial triangulation of $B_0$. Assume that $\tau_0$ is fine enough so that projections of the singular fibers lie in the interiors of distinct disjoint triangles which do not meet the boundary of $B_0$. Each $N \times I$ component of $M_1$ projects down to a rectangle $R$ with one edge $c$ given by the projection of the singular annulus of $N \times I$, two edges $b_1$, $b_2$ given by the projection of $N \times \del I$ and an edge $a$ given by the projection of $\del N \times I$ to $\del B_0$. Attach one such rectangle to $\tau_0$ for each $N\times I$ component of $M_1$ by introducing the $2$-cell $R$, the edges $b_1$, $b_2$ and $c$ and the four corner vertices  (see Figure \ref{fig: CWcomplex}). Similarly, each $N \times S^1$ or $N\tilde{\times} S^1$ component of $M_1$ projects down to an annulus $A$. Give a cell-structure to $A$ by introducing an edge $b$ corresponding to the projection of $N \times t_0$ for some fixed $t_0 \in S^1$, a boundary edge-loop $c$ corresponding to the projection of the singular torus (in $N \times S^1$) or singular Klein-bottle (in $N \tilde{\times} S^1$) and a boundary edge-loop $a$ that lies in $\del B_0$.  Let $R$ denote the $2$-cell $A \setminus \{a, b, c\}$. For each $N\times S^1$ and $N \tilde{\times} S^1$ component of $M_1$, insert such a $2$-cell $R$, the edge $b$ and $c$ and the two end-vertices of $b$ (see Figure \ref{fig: CWcomplex}). Call all the edges corresponding to $b$, $b_1$ and $b_2$ as Mobius edges because their pre-images in $M$ are Mobius strips and call all the edges corresponding to $c$ as singular surface edges as their pre-images in $M$ are singular surfaces. Let $\tau$ be the cell structure of $B$ obtained by attaching these cells to $\tau_0$. 

We will henceforth ignore the singular-surface edges of $\tau$. Let $\mathcal{V}$, $\mathcal{E}$ and $\mathcal{F}$ denote the collection of preimages of the vertices, the edges which are not singular surface edges and the faces of $\tau$ respectively. Each $V \in \mathcal{V}$ is a fiber of $M$. Each $E \in \mathcal{E}$ is a vertical Mobius strip when $E$ is a Mobius edge and is a vertical annulus otherwise. And  each $F\in \mathcal{F}$ is a model regular solid torus when it does not contain any singular fibers, is a model non-regular solid torus when it contains an isolated singular fiber and is a model solid Klein bottle $N \times I$ otherwise. 

\section{Incompressible surfaces in Seifert fiber spaces}
In this section, we study the incompressible surfaces in Seifert fiber spaces and prove Theorem \ref{thm: main theorem}. The main original contribution here is a list of the incompressible surfaces in a solid Klein bottle given in Theorem \ref{thm: incompressible in solid Klein bottle}. The rest of this section extends the proof of Rannard \cite{Ran} in order to apply it to the case when $M$ may have singular surfaces and non-empty boundary.

\begin{definition}
Let $S$ be a properly embedded surface in $M$ that intersects all the annuli and mobius strips $E \in \mathcal{E}$ transversely. Define the complexity $\xi(S)$ of $S$ to be $\sum_{E \in \mathcal{E}} |\pi_0(S \cap E)|$. We say that $S$ is minimal if it has minimal complexity in its isotopy class. We say that $S$ is well-embedded if it is minimal and it intersects each $E\in \mathcal{E}$ horizontally or vertically.
\end{definition}

Applying the below Lemma \ref{lem: well-embedded} to each $E\in \mathcal{E}$ shows that every essential surface in $M$ can be isotoped to a well-embedded surface. This Lemma \ref{lem: well-embedded} combines the arguments of Lemma 3.1 and Lemma 3.2 of Rannard \cite{Ran} and extends it to manifolds with boundary.

\begin{lemma}\label{lem: well-embedded}
Let $S$ be a connected essential minimal surface in $M$. For each $E_0\in \mathcal{E}$ there exists an isotopy of $S$ in an arbitrarily small neighbourhood of $E_0$ that fixes $\del E_0$ and takes $S$ to a minimal surface which intersects $E_0$ horizontally or vertically.
\end{lemma}
\begin{proof}
We study the various possibilities for the intersection of the surface $S$ with the annulus or Mobius strip $E_0$. \newline
\emph{Case I: A component of $S\cap E_0$ is a closed curve that bounds a disk in $E_0$.} By taking the innermost such disk $D$ in $E_0$ we may assume that $S$ does not intersect the interior of $D$. As $S$ is incompressible, there exists a disk $D_0 \subset S$ with $\del D=\del D_0$. As $M$ is irreducible, $D \cup D_0$ bounds a ball in $M$. When $E_0$ does not lie in $\del M$, we may reduce $\xi(S)$ by isotoping $D_0$ across this ball and off $D$. This contradicts the minimality of $S$. If $E \subset \del M$ then as $S$ is connected $D_0=S$ is a boundary parallel disk, which contradicts the fact that $S$ is essential. So we may assume that no component of $S\cap E_0$ is a trivial closed curve in $E_0$.
  
\emph{Case II: A component of $S \cap E_0$ is a boundary-parallel arc in $E_0$.} By taking the outermost disk $D$ cut off by such arcs, we may assume that the interior of $D$ is disjoint from $S$. Let $\beta$ be the arc $D \cap \del E_0$ which lies in a component $V$ of $\del E_0$. 

If $V$ lies in the interior of $M$ or if $E_0 \subset \del M$, then we claim that isotoping $S$ across $D$ and off $\beta$ reduces $\xi(S)$. Let $E \in \mathcal{E}$ be such that $E \neq E_0$ and $V \subset \del E$. Let $\alpha_1$ and $\alpha_2$ be components of $S \cap E$ that meet the end-points of $\alpha$ ($\alpha_1$ may be equal to $\alpha_2$). After this isotopy, $\{\alpha_1,  \alpha_2\}$ is replaced in the list of connected components of $S\cap E$ by the component $\alpha_1 \cup \beta \cup \alpha_2$. So the number of components of $S\cap E$ does not increase under such an isotopy when $V \subset \del E$. If $E\in \mathcal{E}$ is such that $V$ does not lie in $\del E$, then this local isotopy does not change $S \cap E$. And lastly, this isotopy reduces the number of components of $S\cap E_0$. Therefore, the complexity $\xi(S)$ of $S$ reduces after this isotopy, contradicting the minimality of $S$.

Assume that $V$ lies in $\del M$ and $E_0$ does not lie in $\del M$. The disk $D$ can not be a boundary-compressing disk for $S$, which is given to be essential. So $D\cap S$ is an arc $\alpha$ that is boundary-parallel in $S$. Hence there exists a disk $D_0 \subset S$ with $\del D_0 = \alpha \cup \gamma$, where $\gamma$ is an arc in $\del S$. As $\del M$ is incompressible in $M$, so there exists a disk $D_1 \subset \del M$ such that $\del D_1 = \del (D \cup D_0) = \beta \cup \gamma$. As $\beta$ is vertical, so there exists an $E_1 \in \mathcal{E}$ that lies in $\del M$ such that $\gamma$ cuts off a disk $D'_1 \subset D_1$ from $E_1$. By the above argument applied to $S\cap E_1$ isotoping $S$ across $D'_1$ reduces the complexity of $S$, contradicting the minimality of $S$. So we may assume that no component of $S \cap E_0$ is a boundary-parallel arc in $E_0$.

\emph{Case III: A component of $S \cap E_0$ is an essential closed curve in $E_0$.}
By the above cases, there are no trivial closed curves or boundary-parallel arcs in $S \cap E_0$. If some component of $S \cap E_0$ is an essential closed curve in the annulus or Mobius strip $E_0$, then $S\cap E_0$ is a union of essential closed curves. There is an isotopy of $E_0$ fixing $\del E_0$ which takes these closed curves to fibers of $E_0$. We can extend this isotopy of $E_0$ to a local isotopy of an arbitrarily small neighbourhood of $E_0$, which fixes $\del E_0$ and does not increase $\xi(S)$.

\emph{Case IV: No component of $S \cap E_0$ is a closed curve or a boundary-parallel arc in $E_0$.}
If there are no closed curves and no boundary-parallel arcs in $S\cap E_0$, then there exists a local isotopy of $E_0$ fixing $\del E_0$ which takes $S \cap E_0$ to a union of horizontal arcs. Again, we can extend this isotopy of $E_0$ to a local isotopy in an arbitrarily small neighbourhood of $E_0$ which fixes $\del E_0$ and does not increase $\xi(S)$.
\end{proof}

We now focus on the intersection of an essential surface with the solid tori and solid Klein-bottles in $\mathcal{F}$. The following lemmas lead up to Theorem \ref{thm: incompressible in solid torus} and Theorem \ref{thm: incompressible in solid Klein bottle}. The below Lemma \ref{lem: annuli are vertical} is an extension of Lemma 3.6 of Rannard \cite{Ran} to include boundary-parallel Mobius strips in solid Klein bottle.

\begin{lemma}\label{lem: annuli are vertical}
Let $S$ be a well-embedded essential surface in $M$. Fix $F \in \mathcal{F}$ and let $S_0$ be a component of $S \cap F$. Assume that $S_0$ is a boundary-parallel annulus or Mobius strip in $F$. Then $S_0$ can be isotoped fixing $\del F$ to a vertical surface.
\end{lemma}
\begin{proof}
As $S_0$ is boundary-parallel, there exists a surface $S'_0 \subset \del F$ isotopic to $S_0$ in $F$ with $\del S'_0= \del S_0$. As $S$ is well-embedded, $\del S_0$ is either horizontal or vertical in $\del F$. If $\del S_0$ is vertical, then $S'_0$ is a vertical annulus or Mobius strip in $\del F$. Pushing the interior of $S'_0$ into the interior of $F$ we get a properly-embedded vertical surface as required. Assume henceforth that $\del S_0= \del S'_0$ is horizontal in $\del F$. 

If $S'_0 \subset \del F$ is a Mobius strip, then $F$ is a solid Klein bottle and as $\del S'_0$ is an embedded closed curve transverse to the fibration of the Klein bottle, so it must be isotopic to $m$. But as $m$ is non-zero in $H_1(\del F, \Z_2)$, so there can not exist a Mobius strip in $\del F$ with boundary isotopic to $m$. 

If $S'_0 \subset \del F$ is an annulus, then there exists a solid torus $Q$ in $F$ with $\del Q = S_0 \cup S'_0$. As $\del S'_0$ is horizontal, so it intersects each $E \in \mathcal{E}$ that lies in $\del F$ in horizontal arcs. As there is more than one cell in the cell-structure $\tau$ of the base space $B$, so  $\del F$ is not a subset of $\del M$. Take some $E_0\in \mathcal{E}$ that lies in $\del F$ and does not lie in $\del M$. As $\del S'_0 \cap E_0$ is a union of horizontal arcs so there exists a vertical arc $\beta \subset S'_0 \cap int(E_0)$ with $\beta \cap \del S_0 = \del \beta$. Let $\alpha$ be the arc in $S_0$ isotopic relative boundary to $\beta$ in $Q$. Let $D$ be the meridian disk in the solid torus $Q$ with $\del D = \alpha \cup \beta$. Let $\delta_1$ and $\delta_2$ be the horizontal arcs of $S_0 \cap E_0$ which contain $\del \beta$. Let $\beta \times I$ denote a thickening of $\beta$ in $E_0$ with $\del \beta \times I \subset \delta_1 \cup \delta_2$. 
Isotoping $S$ across $D$ and off $\beta$ replaces the two components $\delta_1$ and $\delta_2$ of $S\cap E_0$ with the two components given by $(\delta_1 \cup \delta_2 \cup \beta \times \del I) \setminus (\del \beta \times int(I))$. These arcs cut off disks from $E_0$, however the complexity of $S$ does not change under such a local isotopy. As in Lemma \ref{lem: well-embedded} Case II, by isotoping $S$ off such a disk cut off from $E_0$ the complexity $\xi(S)$ of $S$ can be reduced, contradicting the minimality of $S$.
\end{proof}

Any pair of disjoint essential curves on a torus are isotopic curves. This is however not true on a Klein bottle. The below Lemma \ref{lem: boundary-parallel annulus} says that if the boundary components of a boundary-compressible incompressible surface are isotopic then the surface must be a boundary-parallel annulus.

\begin{lemma}\label{lem: boundary-parallel annulus}
Let $T$ be a torus or Klein bottle boundary component of $M$. Let $S$ be a connected incompressible surface in $M$ and let $\gamma_1$ and $\gamma_2$ be distinct components of $S \cap T$. Let $D$ be a boundary compressing disk of $S$ with $D \cap S$ an arc with endpoints on both  $\gamma_1$ and $\gamma_2$. If $\gamma_1$ is isotopic to $\gamma_2$ in $T$, then $S$ is a boundary-parallel annulus.
\end{lemma}

\begin{proof}
If a component of $\del S$ is trivial in $T$, then let $D_0$ be the innermost disk bounded by $\del S$ in $T$. As $S$ is incompressible, there exists a disk $D_1 \subset S$ with $\del D_1 = \del D_0$. As $S$ is connected $S=D_1$. And as $M$ is irreducible, so the sphere $D_0 \cup D_1$ bounds a ball in $M$. Hence $S$ is parallel to the disk $D_0 \subset \del M$. This contradicts the fact that $\del S$ has at least two components. So we may assume that $\gamma_1$ and $\gamma_2$ are essential curves in $T$.

Let $\del D=\alpha \cup \beta$ where $\alpha$ and $\beta$ are the arcs $D \cap S$ and $D \cap T$ respectively. If $T$ is a torus, then the closure of any component of the complement of any two disjoint essential closed curves in $T$ is an annulus. Assume that $T$ is a Klein bottle and refer to Figure \ref{fig: CurvesKB} for the labels of the closed embedded curves on $T$. Any two curves on a Klein bottle which are isotopic to $l_1$ (or to $l_2$) must intersect. If the curves $\gamma_i$ are isotopic to $m$, then the closure of any component of their complement in $T$ is an annulus. If the $\gamma_i$ are isotopic to $d$, then the closures of their complementary components give two Mobius strips and an annulus.  In either case, as the endpoints of $\beta$ are in distinct components of $S \cap T$, so the closure of the component of $T \setminus \del S$ which contains $int(\beta)$ must be an annulus $A$ with $\del A = \gamma_1 \cup \gamma_2$. Note that if $\gamma_1$ were not given to be isotopic to $\gamma_2$, for example $\gamma_1 = l_1$ and $\gamma_2 = d$, then the $int(\beta)$ may lie in an open annulus but its closure will not remain an annulus.

Let $N(D)= D \times [-1,1]$ be a regular neighbourhood of $D$ in $M\setminus S$. The boundary of $N(D)$ decomposes as $N(\alpha) \cup N(\beta) \cup (D \times \{-1,1\})$. Let $D'$ be the disk obtained by taking the closure of $A \setminus N(\beta)$. Then $D_0=D' \cup (D \times \{-1,1\})$ is an embedded disk with boundary in $S$. As $S$ is incompressible, so there exists a disk $D_1\subset S$  with $\del D_1 = \del D_0$. As $M$ is irreducible, so $D_0 \cup D_1$ bounds a ball $B$ in $M$ with $B \cap N(D) = D \times \{-1, 1\}$. And $S=D_1 \cup N(\alpha)$ is the union of two disks along a pair of arcs $\alpha \times \{-1,1\}$ on their boundary, so it is either an annulus or a Mobius strip. But as $S$ has at least two boundary components so it can not be a Mobius strip. Furthermore, $S$ cuts off the solid torus $B \cup N(D)$  from $M$ so it is boundary-parallel.
\end{proof}

\begin{lemma}\label{lem: intersect meridian disk in arcs}
Let $F$ be a fibered solid torus or fibered solid Klein bottle and let $S$ be a connected incompressible surface in $F$ with horizontal or vertical boundary. Then, there exists a boundary relative isotopy which takes $S$ to a horizontal disk or to a surface that intersects a horizontal disk transversely in a non-empty collection of arcs which are not boundary-parallel in $S$.
\end{lemma}
\begin{proof}
Let $D$ be a horizontal meridian disk of $F$. We claim that there exists a boundary-relative isotopy, which takes $S$ to a surface that does not intersect the interior of $D$ in any circles. To see this, let $D_0$ be an innermost disk in $D$ bounded by the circles in $S \cap D$. As $S$ is incompressible, $\del D_0$ bounds a disk $D_1$ in $S$. By the irreducibility of $F$, $D_0 \cup D_1$ bounds a ball in $F$. Therefore $D_1$ can be isotoped off $D_0$ through this ball to reduce the number of circles in $S\cap D$. The claim then follows from induction on the number of circle components of $S \cap D$ in the interior of $D$. 

If $S$ does not intersect the interior of $D$, then it is an incompressible surface in the ball $B=F \setminus N(D)$, and so it must be a boundary-parallel disk in $B$. If $\del S$ is a trivial closed curve in the fibered annulus $\del D \times I \subset \del B$, then it can not be horizontal or vertical, which contradicts our assumption. So $\gamma=\del S$ separates the two copies of $D$ in $\del B$. Hence, there exists a fiber-preserving isotopy of $F$ which takes $\del D$ to the horizontal curve $\gamma$. Such an isotopy takes $D$ to a horizontal disk $D'$ with $\del D' =\gamma$. Applying the above arguments to $S\cap D'$, we can conclude that $S$ does not intersect the interior of $D'$ in any circles and so $S\cap D' = \del S$. Therefore by incompressibility of $S$ and irreducibility of $M$, $S$ is isotopic relative boundary to the horizontal disk $D'$ in $M$. 

Assume that $S$ intersects $D$ in a non-empty collection of arcs. If any of these arcs is boundary-parallel  in $S$ then choose an outermost such arc $\alpha$ that cuts off a disk $D_0$ from $D$. Let $\del D_0 = \alpha \cup \beta$ with $\beta$ as the arc $D_0 \cap \del F$. As $\alpha$ is boundary-parallel in $S$, so there exists a disk $D_1 \subset S$ such that $\del D_1 = \alpha \cup \gamma$ with $\gamma$ the arc $D_1 \cap \del F$. The embedded disk $D'=D_0 \cup D_1$ is either a meridian disk or a boundary parallel disk of $F$. If $D'$ is a meridian disk, then after a slight perturbation we may assume that it is disjoint from $S$. And so by the arguments above, $S$ lies in the ball $F \setminus N(D')$, and hence is isotopic relative boundary to a horizontal disk. If $D'$ is a boundary parallel disk then it bounds a ball $B'$ in $F$. There exists an isotopy of $F$ that is identity oustide a neighbourhood of $B'$, which sweeps $D_0$ across $B'$ and off $D_1$. Restricting this isotopy to $D$ takes $D$ to a meridian disk with fewer number of components of intersection with $S$.
\end{proof}
 
We state below the kind of non-orientable incompressible surfaces that exist in a solid torus, following the characterisation of such surfaces given by Rubinstein \cite{Rub}.

\begin{lemma}[Corollary 2.2 \cite{Fro}]\label{lem: boundary of non-orientable surface in solid torus}
A one-sided incompressible surface in a solid torus has as boundary a single $(2k, q)$-curve with $k\neq 0$. Conversely every $(2k, q)$-curve with $k \neq q$ on the boundary of a solid torus is the boundary of a one-sided incompressible surface in the solid torus.
\end{lemma}

For the sake of exposition we expand on a proof of the below result that has been proved in Rannard \cite{Ran}.
\begin{lemma}[Lemma $3.7$ \cite{Ran}]\label{lem: reduce to meridian disks}
	Let $S$ be a connected incompressible non-orientable surface in a regular solid torus $F\in \mathcal{F}$. Assume that the boundary of $S$ is horizontal or vertical in $\del F$. Then, $S$ can be reduced to a meridian disk by a sequence of boundary compressions with respect to any $E\in \mathcal{E}$ with $E\subset \del F$.
\end{lemma}
\begin{proof}
By Lemma \ref{lem: boundary of non-orientable surface in solid torus}, $S$ is a connected incompressible surface with connected boundary. If $S$ is not a disk, then we will show that there exists a boundary compression of $S$ along $E$ that increases its Euler characteristic by one while keeping $S$ a connected incompressible surface with connected boundary.

By Lemma \ref{lem: intersect meridian disk in arcs}, there exists a boundary-relative isotopy that takes $S$ to a surface which intersects a horizontal disk $D$ in a non-empty collection of arcs which are not boundary-parallel in $S$. Let $\alpha$ be an outermost arc which cuts off a disk $D_0$ from $D$. Let $\del D_0 = \alpha \cup \beta$ where $\beta$ is the arc $\del F \cap D_0$. Assume that $\beta$ lies in $E$.  Let $N(\del S)$ be an annulus neighbourhood of $\del S$ in $S$ and let $N(\alpha)$ be a rectangular neighbourhood of $\alpha$ in $S$. Translating a normal vector to $S$ pointing into $D_0$ along $\alpha$ from $\alpha(0)$ to $\alpha(1)$ followed by a translation along $\del S$ from $\alpha(1)$ to $\alpha(0)$ reverses the original vector. So the $1$-handle $N(\alpha)$ is attached to the annulus $N(\del S)$ with a twist, giving a once-punctured Mobius strip $Q$. The complement of $\alpha$ in $Q$ is connected and hence its complement in $S$ is also connected. So compressing $S$ along the disk $D_0$ gives the required connected incompressible surface with connected boundary and with Euler characteristic one more than that of $S$.

We will now obtain such a boundary-compressing disk when $\beta$ is not in $E$. We first claim that $\beta$ can be isotoped to lie arbitrarily close to a vertical arc via an isotopy through arcs with endpoints on $\gamma=\del S$. Let $A$ be an annulus such that $\del F$ is obtained by identifying the interior of $A$ with $\del F \setminus \gamma$ and $\del A$ with $\gamma$. $A$ has an $I$-fibering induced by the fibering of $\del F$. Let $\gamma_0$ and $\gamma_1$ denote the boundary components of $A$, containing $\beta(0)$ and $\beta(1)$ respectively.  As $\del D$ and hence $\beta$ is horizontal, so $\beta(0)$ and $\beta(1)$ do not lie on the same $I$-fiber of $A$. Let $\overline{\beta}$ be the projection of $\beta$ on $\gamma_1$. As $\del D$ and hence $\beta$ intersects each fiber of $\del F$ at most once so the vertical $I$-fiber of $A$ at $\beta(0)$ followed by $\overline{\beta}$ followed by the reverse of $\beta$ is an embedded closed curve that bounds a disk in $A$. This gives an isotopy in $A$ taking $\beta$ to the $I$-fiber at $\beta(0)$ via arcs with one endpoint fixed at $\beta(0)$ and the other on $\overline{\beta}$. No point of $\overline{\beta}$ other than possibly $\overline{\beta}(0)$ is identified with $\beta(0)$ in $\gamma$ as $\beta$ intersects each fiber of $\del F$ at most once. If $\overline{\beta}(0)$ is identified with $\beta(0)$ in $\gamma$, then $\gamma$ intersects each fiber of $\del F$ exactly once and so, the two boundary components of $A$ are identified by full twists, otherwise they are identified via some rational twist. In either case, we have an isotopy in $\del F$ that takes $\beta$ close to a vertical arc via arcs with endpoints on $\del S$. Following this up with a rotation in $A$ takes $\beta$ to an arc $\delta$ in $E \cap A$ via some isotopy $H: I \times I \to \del F$.

Let $P = \del F \times I$ be a neighbourhood of $\del F$ in $F$, such that $S \cap P = \del S \times I$, $D \cap P = \del D \times I$, $D_0 \cap P = \beta \times I$. Let $F'$ be the closure of $F \setminus P$, i.e., the inner solid torus. Let $D'_0 = F' \cap D_0$, $\beta' = \del F' \cap D'_0$, $\alpha'= F' \cap \alpha$ and $S' = F' \cap S$. Define $G: I \times I \to \del F \times I = P$ by $G(s, t) = (H(s, t), t)$. $G$ is then an embedding of a rectangle with boundary consisting of a pair of opposite edges $\eta_1 \cup \eta_2$ on $\del S \times I = S \cap P$ and the remaining pair of edges as $\beta' \cup \delta$. Let $\alpha'' = \alpha' \cup \eta_1 \cup \eta_2$  and let $D_1 = D'_0 \cup G(I \times I)$. Then, the disk $D_1$ has boundary $\alpha'' \cup \delta$ with $\alpha''$ the arc $\del D_1 \cap S$ and $\delta$ the arc $\del D_1 \cap \del F$. As $\alpha''$ is isotopic on $S$ to $\alpha$ which is not boundary-parallel in $S$, so $D_1$ is a boundary-compression of $S$ with respect to $E$. 
\end{proof}

\begin{lemma}\label{lem: incompressible pieces}
Let $P$ be a properly embedded $2$-sided surface in $M$ that is a union of some $E\in \mathcal{E}$. Let $S$ be an incompressible minimal surface in $M$. Let $M'$ be the closure of a component of $M \setminus P$. Then, $S\cap M'$ is an incompressible surface in $M'$.
\end{lemma}
\begin{proof}
Suppose that $S'=S\cap M'$ is compressible. Let $D$ be a compressing disk for $S'$ in $M'$. As $S$ is incompressible in $M$, so there exists a disk $E \subset S$ such that $\del D = \del E$. As $\del D$ is essential in $S'$, so $E$ does not lie in $S'$ and must therefore intersect $P$. As $M$ is irreducible, the sphere $D \cup E$ bounds a ball in $M$. Isotoping $E$ across this ball to $D$ reduces the number of components of $S \cap P$. As $P$ is a union of some $E\in \mathcal{E}$, so this isotopy reduces $\xi(S)$, contradicting the fact that $S$ is minimal. Therefore, $S'$ is incompressible in $M'$. 
\end{proof}

We list below the possible incompressible surfaces in a solid torus $F\in \mathcal{F}$. Note that the solid torus may have a non-regular fibering. The following Theorem \ref{thm: incompressible in solid torus} follows from Lemma 3.5 and Lemma 3.6 of Rannard \cite{Ran}. We give a proof here for the sake of exposition and to ensure that the isotopy is boundary-relative.

\begin{theorem}\label{thm: incompressible in solid torus}
Let $S$ be a connected well-embedded essential surface in $M$ and let $F \in \mathcal{F}$ be a solid torus. Then, there is an isotopy of $S$ which pointwise fixes all $E \in \mathcal{E}$, and takes $S\cap F$ to a union of the following components:
\begin{enumerate}
	\item{A vertical boundary-parallel annulus}
	\item{A horizontal disk}
	\item{A once-punctured non-orientable surface whose boundary is neither a meridian nor a longitude of $F$}
\end{enumerate}
\end{theorem}

\begin{proof}
Let $P$ be the union of all $E\in \mathcal{E}$ that lie in $\del F$ and do not lie on $\del M$. By Lemma \ref{lem: incompressible pieces}, $S\cap F$ is an incompressible surface. And as $S$ is well-embedded so the boundary of $S\cap F$ is either horizontal or vertical. Let $S_0$ be a component of $S\cap F$. By Lemma \ref{lem: intersect meridian disk in arcs}, there exists an isotopy which fixes $\del F$ and takes $S_0$ to either a horizontal disk or to a surface that intersects a horizontal disk $D$ in arcs that are not boundary-parallel in $S_0$. Assume that $S_0$ is not isotopic to a horizontal disk.

Let $D_0$ be the outermost disk cut off by arcs of $S_0 \cap D$ in $D$. Let $\del D_0 = \alpha \cup \beta$ where $\alpha = S_0 \cap D_0$ and $\beta = \del F \cap D_0$. Let $\gamma_1$ and $\gamma_2$ be the components of $\del S_0$ that contain $\del \alpha$. If $\gamma_1$ and $\gamma_2$ are distinct curves then by Lemma \ref{lem: boundary-parallel annulus}, $S_0$ can be isotoped fixing boundary to a boundary-parallel annulus. By Lemma \ref{lem: annuli are vertical}, $S_0$ can be isotoped fixing boundary to a  vertical annulus and in particular, $\del S_0$ is a vertical curve. If $\gamma_1 = \gamma_2$, then we claim that $S_0$ is a one-sided surface in a solid torus and is therefore non-orientable. To see this, observe that the normal vector field to $S_0$ along $\alpha$ pointing into $D_0$, followed by the normal vector field along $\gamma_i$ from one end-point of $\alpha$ to the next gives a normal vector field along a closed curve on $S_0$ that reverses the direction of the initial vector. 
By Lemma \ref{lem: boundary of non-orientable surface in solid torus}, the boundary of $S_0$ is connected and is neither a meridian nor a longitude of $F$. $S_0$ is therefore isotopic relative boundary to either a horizontal disk, a vertical boundary-parallel annulus or to a once-punctured non-orientable surface whose boundary is neither a meridian nor a longitude. We shall now see that there exists a boundary-relative isotopy of $F$ which takes all components of $S\cap F$ simultaneously to one of these three types.

As the Dehn surgery slope at an isolated singular fiber of $M$ is finite, so the slope of a fiber of $\del F$ is non-zero even when $F$ is a non-regular solid torus. And by Lemma \ref{lem: boundary of non-orientable surface in solid torus}, the slope of the boundary of a non-orientable surface is also non-zero. Therefore if some component of $S\cap F$ is isotopic relative boundary to a horizontal disk then every component of $S\cap F$ is isotopic relative boundary to a horizontal disk. So a single isotopy fixing $\del F$ exists taking $S\cap F$ to a union of horizontal disks.

Assume that $S \cap F$ is a union of boundary-parallel annuli and non-orientable surfaces. Any two one-sided surfaces in $F$ with isotopic boundary slopes must intersect, so there is at most one such surface. As any boundary parallel annuli separates $F$ into two solid tori, so only one of these pieces can contain the non-orientable surface component. So as in Lemma \ref{lem: annuli are vertical}, there is an isotopy of $F$ which isotopes all the boundary-parallel annuli with vertical boundary into a neighbourhood of $\del F$ where they are vertical.
\end{proof}

We now list the possible incompressible surfaces in a solid Klein bottle. 
\begin{theorem}\label{thm: incompressible in solid Klein bottle}
Let $F=N \times I$ be a fibered solid Klein bottle and let $S$ be an incompressible surface in $F$. Assume that the boundary of $S$ is horizontal or vertical in $\del F$. Then $S$ can be isotoped relative boundary to a union of the following components (see Figure \ref{fig: IncompressibleKB}):
\begin{enumerate}
\item{A horizontal disk.}
\item{A vertical boundary-parallel Mobius strip with boundary $d$.}
\item{A vertical boundary-parallel annulus with boundary two copies of $d$.}
\item{A vertical one-sided annulus with boundary $l_1 \cup l_2$.}
\item{A one-sided pair of pants with boundary $l_1 \cup l_2 \cup d$.}
\end{enumerate}
\end{theorem}
\begin{figure}
\centering
\def\svgwidth{1\columnwidth}
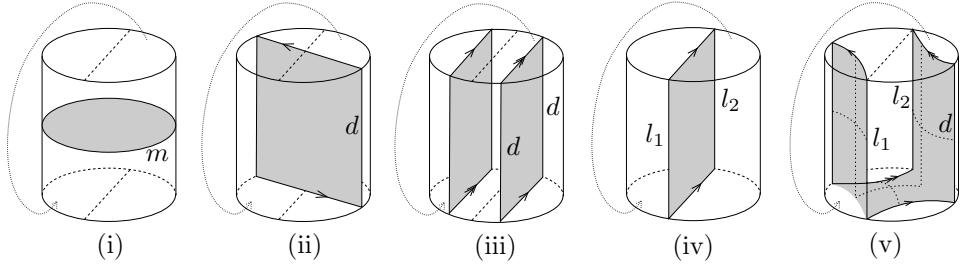
\caption{Incompressible surfaces in a solid Klein bottle represented as $\mathbb{D}^2 \times I$ with the monodromy reflection along a diameter: (i) Horizontal disk with boundary $m$, (ii) Vertical boundary-parallel Mobius strip with boundary $d$, (iii) Vertical boundary-parallel annulus with boundary $d \cup d$, (iv) Vertical one-sided annulus with boundary $l_1 \cup l_2$  and (v) One-sided pair of pants with boundary $l_1 \cup l_2 \cup d$.}\label{fig: IncompressibleKB}
\end{figure}

 \begin{proof} 
Assume that $S$ is connected. If $\del S$ is horizontal, then there exists a horizontal disk $D$ whose boundary is disjoint from $\del S$. And so by Lemma \ref{lem: intersect meridian disk in arcs}, $S$ is isotopic relative boundary to a horizontal disk. (See Figure \ref{fig: IncompressibleKB} (i))

Assume that $\del S$ is vertical in $\del F$. The Klein bottle $\del F$ is the union of two Mobius strips so the boundary of $S$ consists of components that are either the cores $l_1$ or $l_2$ of these Mobius strips or copies of the common boundary $d$ of the Mobius strips. By Lemma \ref{lem: intersect meridian disk in arcs}, $S$ intersects a horizontal disk $D$ in arcs that are not boundary-parallel in $S$.

Let $B = D \times I$ and let $F$ be obtained from $B$ by identifying $D_1=D \times 1$ with $D_0=D \times 0$ via a reflection along the diameter of $D$ joining $l_1$ and $l_2$. We claim that $S'=S\cap B$ is an incompressible surface in $B$. To see this, let $E_0$ be a compressing disk of $S'$ in $B$. After an isotopy we may assume that $E_0$ does not intersect $D_0$, i.e., $E_0$ lies in the interior of $B$. As $S$ is incompressible, so there exists a disk $E_1$ in $S$ with $\del E_1 = \del E_0$. But as $S$ does not intersect $D_0$ in any circles, so interior of $E_1$ is disjoint from $\del B$, i.e. $E_1 \subset S\cap B = S'$, which contradicts the fact that $\del E_0$ is essential in $S'$. Therefore, $S'$ is incompressible in a ball and hence is a disjoint collection of disks.  Note that $S$ can be reconstructed from $S'$ by identifying arcs of $S' \cap D_1$ with their reflections in $S' \cap D_0$. We now analyse the system of arcs in $S' \cap D_1$.

\emph{Case I: $S' \cap D_1$ contains an outermost arc with endpoints on distinct copies of $d$.} Such an outermost arc cuts off a compressing disk $E$ from $D_1$ which satisfies the conditions of Lemma \ref{lem: boundary-parallel annulus}, so $S$ is a boundary-parallel annulus. As $\del S$ is vertical so there exists a boundary-relative isotopy that takes $S$ to a vertical surface. See Figure \ref{fig: IncompressibleKB} (iii).

\emph{Case II: $S' \cap D_1$ contains an outermost arc with both end points on the same curve $d$.} Let $d'$ and $d''$ be the two vertical arcs given by the intersection of $S'$ with $\del D \times I \subset \del B$, so that $d=d' \cup d''$. Let $\alpha_1$ be an arc in $S\cap D_1$ connecting $d'$ and $d''$ and let $\alpha_0$ be its reflection in $S \cap D_0$ which also connects $d'$ and $d''$. Then, $\gamma= \alpha_0 \cup \alpha_1 \cup d' \cup d''$ is a closed curve in $\del S'$. So, $S'$ is a single disk with boundary $\gamma$. But as $\alpha_0$ is identified with $\alpha_1$ via a reflection in $S$ and $S$ is connected, so $S$ is a boundary-parallel Mobius strip with boundary $d$ as in Figure \ref{fig: IncompressibleKB} (ii). Again as $\del S$ is vertical, so there exists a boundary-relative isotopy that takes $S$ to a vertical Mobius strip.

\emph{Case III: $S' \cap D_1$ does not contain any outermost arc with both end points on copies of $d$.}  Any outermost arc of $S' \cap D_1$ must have an end point on either $l_1$ or $l_2$ so there are at most two such arcs. If there is only one outermost arc $\alpha_1$ with end points on both $l_1$ and $l_2$, then its reflection on $D_0$ is an arc $\alpha_0$ that also joins $l_1$ and $l_2$. Arguing as in Case II then, $S'$ is a disk with boundary $\alpha_0 \cup \alpha_1 \cup l_1 \cup l_2$. As $\alpha_0$ is identified with $\alpha_1$ via an orientation-preserving map in $S$ and $S$ is connected, so $S$ is an annulus with boundary components $l_1 \cup l_2$, as in Figure \ref{fig: IncompressibleKB} (iv). The complement of $l_1 \cup l_2$ in $\del F$ is connected so such an annulus is one-sided.

\begin{figure}
\centering
\def\svgwidth{0.6\columnwidth}
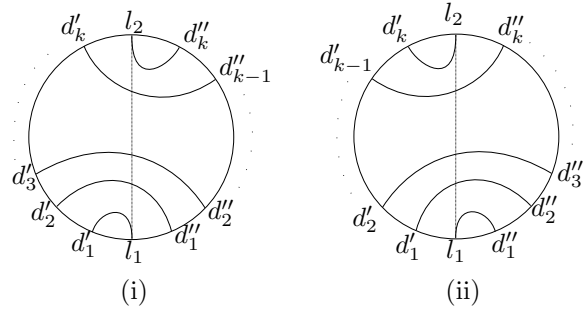
\caption{The pattern of arcs in (i) $U=S\cap D_1$ and its reflection (ii) $L=S \cap D_0$}\label{fig: Arcpattern}
\end{figure}

Assume that there are two outermost arcs in $S' \cap D_1$, one of which has an end point on $l_1$ and the other has an end point on $l_2$. Let $\del S = l_1 \cup l_2 \cup (\cup_{i=1}^k d_i)$ for some $k\geq 1$. Then the pattern of arcs $U$ on $D_1$ is as in Figure \ref{fig: Arcpattern} (i) with outermost arcs from a point in $l_1$ to $d_1$ and $l_2$ to $d_k$, and parallel arcs from points in $d_i$ to $d_{i+1}$. The pattern of arcs $L$ on $D_0$ is a reflection across the diameter $l_1 l_2$ of the pattern $U$ (see Figure \ref{fig: Arcpattern} (ii)). Each boundary component $d_i$ of $S$ splits into two vertical arcs in $S'$, with the vertical arc on the left of the diameter $l_1 l_2$ in $D_1$ denoted by $d'_i$ and the vertical arc on the right by $d''_i$ and ordering of the $d'_i$ and $d''_i$ given via a path on $\del D_1$ from $l_1$ to $l_2$. The boundary of $S'$ alternates between the horizontal arcs in $U$, the vertical arcs $l_1$, $l_2$, $d'_i$ or $d''_i$ and the horizontal arcs in $L$. 

To see that $\del S'$ is connected we trace one such curve starting with $l_1$. If $k$ is odd, then starting from $l_1$, $\del S'$ traces the following vertical arcs in the given order $l_1$, $d'_1$, $d''_2$, $d'_3$, $d''_4$, ..., $d'_k$, $l_2$, $d''_{k}$, $d'_{k-1}$, $d''_{k-2}$, ..., $d''_1$ and back to $l_1$. If $k$ is even, then starting from $l_1$, it traces out the vertical arcs $d'_1$, $d''_2$, $d'_3$, $d''_4$, ..., $d''_k$, $l_2$, $d'_{k}$, $d''_{k-1}$, $d'_{k-2}$, ..., $d''_1$ and back to $l_1$. In either case it runs through all the vertical arcs of $\del S'$ and it is therefore connected. Hence, $S'$ is a single disk.

\begin{figure}
\centering
\def\svgwidth{0.3\columnwidth}
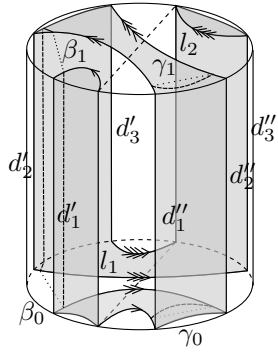
\caption{Shaded disk $S'$ in $B$ containing the vertical arcs of $l_1$, $l_2$, $d_1$, $d_2$ and $d_3$ (i.e., $k=3$). The dashed curve represents the boundary of a compressing disk of $S$.}\label{fig: CompressibleKB}
\end{figure}

Let $E$ be the disk in the sphere $\del B$ bounded by $\del S'$ which contains the disk cut off from $D_1$ by the outermost arc $l_1 d'_1$ (see Figure \ref{fig: CompressibleKB}). Note that $S'$ is a disk in $B$ parallel to $E$. To reconstruct $S$ from $E$ we need to push the interior of $E$ into the interior of $B$ and identify $U$ with $L$, i.e., identify the horizontal arc $l_1 d'_1$ in $U$ with its reflection $l_1 d''_1$ in $L$, the arcs $d''_i d'_{i+1}$ in $U$ with $d'_{i} d''_{i+1}$ in $L$ and the arc $d''_k l_2$ with $d'_k l_2$ if $k$ is odd ($d'_k l_2$ with $d''_k l_2$ if $k$ is even). 

When $k=1$, then $S$ is a pair of pants with boundary $l_1 \cup l_2 \cup d$. To see this, observe that $E$ is an octagon with $\del E$ composed of the arcs $l_1$, $l_1 d'$ in $U$, $d'$, $d' l_2$ in $L$, $l_2$, $l_2 d''$ in $U$, $d''$ and $d'' l_1$ in $L$ (see Figure \ref{fig: IncompressibleKB} (v)). Identifying the horizontal arcs $l_1 d'$ with $l_1 d''$ and $d' l_2$ with $l_2 d''$ gives a pair of pants. Any simple closed curve in a pair of pants is isotopic to one of its boundary components. But as $l_1$, $l_2$ and $d$ (the boundary components of $S$) are all fibers and therefore non-trivial in $F$, so $S$ does not have any compressing disk in $F$.

We shall show that if $k>1$, then $S$ does in fact have a compressing disk which contradicts the incompressibility of $S$. As before, let $E$ be the disk in $\del B$ parallel to $S'$ (see Figure \ref{fig: CompressibleKB}). Let $G$ be the disk in $\del D \times I$ cut off by $d'_1 \cup d'_2$ and $\del D \times \{0, 1\}$. The boundary of $G$ consists of the vertical arcs $d'_1$, $d'_2$ and the horizontal arcs $\beta_i$ in $\del D \times \{i\}$ connecting $d'_1$ and $d'_2$ (which do not pass through $l_i$). Observe that the arcs $\gamma_i$ joining $d''_1$ and $d''_2$ in $\del D_i$ (which do not pass through $l_1$ and $l_2$) lie in $E$. As $S'$ is parallel to $E$, we may assume that $\gamma_1$ and $\gamma_2$ lie in $S'$. After identifying $D_1$ with $D_0$ in $F$, $\gamma_i$ is identified with $\beta_i$ and so $G \cap S = \del G$. To see that $\del G$ is essential in $S$ observe that it intersects the horizontal arc $l_1 d'_1$ in $U$ exactly once. As it intersects a properly embedded arc of $S$ exactly once, so it can not bound a disk in $S$.
\newline

Suppose that $S$ is not connected. We have shown that each component of $S$ is isotopic to one of the five possibilities listed in the statement of this Lemma. We now argue that a single isotopy can be used to simultaneously make horizontal or vertical all the components of $S$ which are not a pair of pants.

If any component of $S$ has horizontal boundary, then it intersects any vertical curve. So if one component of $S$ is isotopic to a horizontal disk then $S$ is a union of disks that can be isotoped to be horizontal. Hence, there exists a single isotopy of $F$ that makes all components of $S$ horizontal. 

Assume that no component of $S$ has horizontal boundary. If a component $C$ of $S$ is a boundary-parallel annulus with boundary two copies of $d$, then its complement in $F$ is a solid Klein bottle containing $l_1$ and $l_2$ and a solid torus $T$. The only components of $S$ that can lie in $T$ are other annuli parallel to $C$. Assume that $C$ is the innermost such annuli. There is an isotopy defined in a neighbourhood of $T$ which takes $C$ into $\del F$. Push the interiors of these annuli back into the interior of $F$ to get an isotopy that fixes all other components of $S$ and takes all the annuli in $T$ to vertical surfaces. Repeating this process for all boundary-parallel innermost annuli we get an isotopy fixing $\del F$ and taking all boundary-parallel annuli to vertical surfaces. 

We now take the solid Klein bottle $F'$ in the complement of these boundary-parallel annuli. And let $C$ be a component of $S$ that is a boundary-parallel Mobius strip in $M'$. Such a component intersects any surface whose boundary contains $l_1$ and $l_2$, so the only remaining components of $S$ in $F'$ are similar boundary parallel Mobius strips. A single isotopy of $F'$ exists that takes all such components to vertical surfaces while fixing the boundary of $F'$. 

As the one-sided annulus and one-sided pair of pants have intersecting boundaries, so there can be at most one such component. If a component $C$ of $S$ in $F'$ is a one-sided annulus with vertical boundary, then there is an isotopy of $F'$ fixing the boundary which takes $C$ to a vertical surface. Combining all these isotopies gives an isotopy which takes $S$ to a vertical surface.
 \end{proof}
 
\begin{corollary}\label{cor: incompressible in Mobius fibered}
Let $M$ be the fibered manifold $N \times S^1$ or $N \tilde{\times} S^1$. Let $S$ be a connected incompressible surface in $M$. Assume that for some $t_0 \in S^1$, $S \cap(N \times t_0)$ is horizontal or vertical and that $\del S$ is either horizontal or vertical in $\del M$. Then $S$ can be isotoped fixing $\del S$ to be either horizontal or vertical or a one-sided once-punctured torus with vertical boundary when $M=N \times S^1$ or a one-sided once-punctured Klein bottle with vertical boundary when $M=N \tilde{\times} S^1$.
\end{corollary}
\begin{proof}
Let $N$ be a fibered Mobius strip $I \times I/ (x, 0) \sim (1-x, 1)$ and let $r$ denote the reflection of $N$ along the arc $I \times \frac{1}{2}$. The manifold $M$ can be obtained from $K=N \times I$ by identifying $N_1 = N \times 1$ with $N_0 = N \times 0$ via the identity if $M = N \times S^1$, or via the reflection $r$ if $M = N \tilde{\times} S^1$. In either case, the fibration on $\del M$ is induced by the circles $\del N \times t$. We may assume that $S$ is either disjoint from or transversely intersects the mobius strip $N \times t_0$, which we identify with $N=N_0=N_1$.

Let $S'= S \cap K$. Assume that $S'$ has a compressing disk $D$ in $K$. As $S$ is incompressible, so there exists a disk $D_0 \subset S$ with $\del D_0 = \del D \subset int(K)$. As $\del D$ is non-trivial in $S'$, so $D_0$ intersects $N$ in some closed curves. This contradicts the fact that $S \cap N$ is horizontal or vertical. Therefore $S'$ is incompressible in $K$.

The fibration of the Klein bottle $\del K$ (given by the fibration of $\del M$ and $N$) is a union of curves parallel to $d$, the curve $l_1$ and the curve $l_2$. By assumption, $\del S'$ is either horizontal or vertical in $\del K$. Applying Theorem \ref{thm: incompressible in solid Klein bottle} to $S'$, we know that each component of $S'$ is either a horizontal disk (meridian disk of the solid Klein bottle $K$), vertical annulus, vertical Mobius strip or a pair of pants with boundary curves $l_1 \cup l_2 \cup d$. 

A horizontal disk of $K$ intersects any surface in $K$ which has vertical boundary curves, so if a component of $S'$ is horizontal then $S'$ is a collection of horizontal disks attached in pairs along arcs on their boundaries. As it is a covering space of an annulus which is the base space of $M$ so $S'$ must be a horizontal annulus.

If $S'$ is a collection of vertical annuli or Mobius strips then $S$ is obtained from $S'$ by attaching them along their boundaries. So $S$ is a vertical annulus, Mobius strip, torus or Klein bottle. 

At most one component of $S'$ is a pair of pants. Assume that some component $C$ of $S'$ is a pair of pants. The boundary components $l_1$ and $l_2$ of $C$ are identified in $S$ when $N_1$ is stuck to $N_0$ via the identity or via the reflection map $r$. If the third boundary component $d$ of $C$ lies on $\del M$, then $S$ is obtained from a pair of pants by identifying two of the boundary curves $l_1$ and $l_2$ via identity if $M=N \times S^1$ and reflection if $M=N \tilde{\times} S^1$. $S$ is therefore a one-sided punctured torus with vertical boundary when $M= N\times S^1$ and a one-sided punctured Klein bottle with vertical boundary when $M=N \tilde{\times} S^1$.

Assume that the third boundary component $d$ of $C$ lies on $N$. A boundary parallel Mobius strip in $K$ separates $l_1$ and $l_2$ and therefore must intersect $C$. So by Theorem \ref{thm: incompressible in solid Klein bottle}, the component of $S'$ that meets the pair of pants  along the curve $d$ on $N$ must be a vertical annulus with boundary two copies of $d$.  If both these boundary components lie on $N$, then by the same reasoning it is adjacent to another vertical annulus. Repeating this argument finitely many times, we get a vertical annulus with one boundary curve $d$ attached to the boundary of the pair of pants and the other boundary curve a fiber of $\del M$. $S$ is therefore obtained from a pair of pants by identifying two of the boundary curves $l_1$ and $l_2$ via identity if $M=N \times S^1$ and reflection if $M=N \tilde{\times} S^1$ and by identifying the third boundary curve to a boundary component of a vertical annulus. And so again, $S$ is a one-sided punctured torus with vertical boundary when $M=N \times S^1$ and a one-sided punctured Klein bottle with vertical boundary when $M= N \tilde{\times} S^1$.
\end{proof}

We now prove Theorem \ref{thm: main theorem} by replicating the proof of Theorem 4.1 of Rannard \cite{Ran}, with modifications to take into account singular surfaces and boundary components.

\begin{proof}
As $M$ contains an essential surface so it is not $S^2 \times S^1$, $S^2 \tilde{\times} S^1$ or $\RP^3 \# \RP^3$. So $M$ is a solid torus, solid Klein bottle or an irreducible manifold with incompressible boundary. Isotope $S$ to have minimal complexity in its isotopy class. After a further isotopy of $S$ near each $E \in \mathcal{E}$ using Lemma \ref{lem: well-embedded}, we may assume that it is well-embedded. Let $M_1$ denote the disjoint union of model neighbourhoods of the singular surfaces and let $M_0 = \overline{M \setminus M_1}$. If $M_0$ is empty, then $M$ is a either $N \times I$, $N\times S^1$ or $N \tilde{\times} S^1$ and so by Corollary \ref{cor: incompressible in Mobius fibered}, $S$ can be isotoped to be pseudo-horizontal or psuedo-vertical. Assume that $M_0$ is non-empty.

When $M_1$ is non-empty, let $P=\del M_0 \cap \del M_1$ be a properly embedded surface that is a union of some $E \in \mathcal{E}$. So by Lemma \ref{lem: incompressible pieces}, $S \cap M_0$ is an incompressible surface in $M_0$. 

By Theorem \ref{thm: incompressible in solid torus}, for each solid torus $F \in \mathcal{F}$, there exists an isotopy pointwise fixing $\del F$ and taking $S \cap F$ to either a union of vertical annuli, a union of horizontal disks or a once-punctured non-orientable surface. Combining these isotopies we get an isotopy of $S$ that pointwise fixes $\mathcal{E}$ and takes $S$ to a well-embedded surface that intersects each solid torus $F \in \mathcal{F}$ in vertical annuli, horizontal disks or a once-punctured non-orientable surface.

\emph{Case I: For some solid torus $F \in \mathcal{F}$, a component of $S \cap F$ is a boundary-parallel vertical annulus.}

As $S$ is well-embedded, it intersects $\del F$ in a union of fibers. Suppose there exists a solid torus $F'$ that is adjacent to $F$ along some $E \in \mathcal{E}$ that intersects $S$. If $F'$ is a regular solid torus, then by Theorem \ref{thm: incompressible in solid torus} and the fact that fibers are longitudes of regular solid tori, $S\cap F'$ must also be a union of vertical annuli. If $F'$ is a non-regular solid torus then as slopes at singular fibers can not be infinite so fibers can not be the boundary of a meridian disk. Therefore, $S\cap F'$ is either a union of vertical annuli or a once-punctured non-orientable surface. So $S \cap (F \cup F')$ is a pseudo-vertical surface. Repeating this argument for adjacent solid tori we can conclude by induction that $S\cap M_0$ is a pseudo-vertical surface with vertical boundary in $\del M_0$ (as non-regular solid tori do not intersect $\del M_0$).

\emph{Case II: For all solid tori $F \in \mathcal{F}$, $S \cap F$ consists of horizontal disks and once-punctured non-orientable surfaces.}

Fix a solid torus $F_0 \in \mathcal{F}$. Suppose that for some regular solid torus $F \in \mathcal{F}$, $S\cap F$ is a punctured non-orientable surface. There is a path of regular solid tori from $F$ to $F_0$. We will use Lemma \ref{lem: reduce to meridian disks} repeatedly to reduce the intersection of $S$ with each solid tori in this path (except possibly $F_0$) to meridian disks.

Let $F' \in \mathcal{F}$ be a regular solid torus adjacent to $F$ along some $E \in \mathcal{E}$. By Lemma \ref{lem: reduce to meridian disks}, we may compress $S$ along $E$ finitely many times to reduce $S\cap F$ to a meridian disk. These compressions give local isotopies that change $S$ only in the interior of $F \cup F'$. As these isotopies fix $S \cap \mathcal{V}$, so the complexity $\xi(S)$ does not change, and so $S$ is still of minimal complexity. By Lemma \ref{lem: well-embedded}, after an isotopy near $E$ in $F \cup F'$, $S$ is a well-embedded surface. Repeat this process along a path of regular solid tori from $F$ to $F_0$. Eventually the surface $S$ is isotoped to a well-embedded surface that intersects each regular solid torus in meridian disks. As $S$ intersects $\del F_0$ horizontally so by Theorem \ref{thm: incompressible in solid torus}, $S \cap F_0$ must be a non-orientable surface or a union of meridian disk. 

If $M$ has an isolated singular fiber, then take $F_0$ to be a non-regular solid torus containing such a fiber, so that $S\cap M_0$ becomes a pseudo-horizontal surface in $M_0$ with horizontal boundary in $\del M_0$. If $M$ has no isolated singular fibers then by assumption it must have singular surfaces, i.e, $M_1$ is non-empty. Take $F_0$ to be a regular solid torus that intersects $\del M_1$ along some $E_0 \in \mathcal{E}$. Assume that $S\cap F_0$ is a once-punctured non-orientable surface. By Lemma \ref{lem: reduce to meridian disks}, repeated boundary compressions of $S\cap F_0$ along $E_0$ reduces $S \cap F_0$ to a meridian disk and such an isotopy does not change $\xi(S)$. Again the surface $S\cap M_0$ has been isotoped to a pseudo-horizontal surface in $M_0$ with horizontal boundary in $\del M_0$.
\newline

If $M_1$ is empty, then we have shown that $S$ is a pseudo-horizontal or pseudo-vertical surface. When $M_1$ is non-empty, take $P = \del M_0 \cap \del M_1$ and by Lemma \ref{lem: incompressible pieces},  $S\cap M_1$ is an incompressible surface. By Lemma \ref{lem: well-embedded}, there is an isotopy in a neighbourhood of $E \in \mathcal{E}$ which makes $S$ a well-embedded surface.  By Theorem \ref{thm: incompressible in solid Klein bottle} and Corollary \ref{cor: incompressible in Mobius fibered}, when $S \cap P$ is vertical as in Case I, for each component  $W$ of $M_1$, $S \cap W$ is either vertical, a pair of pants (when $W=N \times I$), a once-punctured torus (when $W=N \times S^1$) or a once-punctured Klein bottle (when $W=N \tilde{\times} S^1$). Therefore, $S$ is a pseudo-vertical surface in $M$.
And when $S \cap P$ is horizontal as in Case II, then $S\cap M_1$ must be a horizontal surface. And therefore, $S$ is a pseudo-horizontal surface in $M$. 
\end{proof}

\bibliographystyle{alpha}
\bibliography{IncompressibleSFS}

\end{document}